\documentclass[12pt]{amsart}
\usepackage{amssymb}
\newtheorem{lemme}{Lemma}

\newtheorem{theorem}{Theorem}
\newtheorem{remark}{Remark}

\begin{document}

\title [Generalized Huygens types inequalities for Bessel and modified Bessel functions]{Generalized Huygens types inequalities for Bessel and modified Bessel functions\\}%

\author[K. Mehrez]{ Khaled Mehrez }
 \address{Khaled Mehrez. D\'epartement de Math\'ematiques ISSAT Kasserine, Tunisia.}
 \email{k.mehrez@yahoo.fr}
\begin{abstract}
In this paper, we present a generalization of the Huygens types inequalities involving Bessel and modified Bessel functions of the first kind.
\end{abstract}
\maketitle
%%%%%%%%%%%%%%%%%%%%%%%%%%%%%%%%%%%%%%%%%%%%%%%%%%%%%%%


\section{\textbf{Introduction}}
This inequality 
\begin{equation}\label{0001}
2\frac{\sin x}{x}+\frac{\tan x}{x}>3
\end{equation}
which holds for all $x\in(0,\pi/2)$  is known in literature as Huygens's inequality \cite{H}. The hyperbolic counterpart of (\ref{0001}) was established in \cite{N} as follows:
\begin{equation}\label{0002}
2\frac{\sinh x}{x}+\frac{\tanh x}{x}>3,\;x>0. 
\end{equation}
The inequalities (\ref{0001}) and (\ref{0002}) were respectively refined in \cite{H} as 
\begin{equation}
2\frac{\sin x}{x}+\frac{\tan x}{x}>2\frac{x}{\sin x}+\frac{x}{\tan x}>3
\end{equation}
for $0<x<\frac{\pi}{2}$ and 
\begin{equation}
2\frac{\sinh x}{x}+\frac{\tanh x}{x}>2\frac{x}{\sinh x}+\frac{x}{\tanh x}>3,\;x\neq0.
\end{equation}
In \cite {zhu7}, Zhu give some new inequalities of the Huygens type for circular functions,
hyperbolic functions, and the reciprocals of circular and hyperbolic functions, as follows:\\
\noindent\textbf{Theorem A} The following inequalities 
\begin{equation}\label{sw01}
(1-p)\frac{\sin x}{x}+p\frac{\tan x}{x}>1>(1-q)\frac{\sin x}{x}+q\frac{\tan x}{x}
\end{equation}
holds for all $x\in(0, \pi/2)$ if and only if $p\geq1/3$ and $q\leq 0.$\\
\textbf{Theorem B}
 The following inequalities 
\begin{equation}\label{sw03}
(1-p)\frac{\sinh x}{x}+p\frac{\tanh x}{x}>1>(1-q)\frac{\sinh x}{x}+q\frac{\tanh x}{x}
\end{equation}
holds for all $x\in(0, \infty)$ if and only if $p\leq1/3$ and $q\geq 1.$\\

Recently, the author of this paper extend and sharpen inequalities (\ref{sw01}) and (\ref{sw03}) for the Bessel and modified Bessel functions to the following results in \cite{KH}.\\
\noindent\textbf{Theorem C} Let $-1<\nu\leq0$ and  let $j_{\nu,1}$ the first positive zero of the Bessel function $J_\nu$ of the first kind. Then the Huygens type inequalities
\begin{equation}\label{050}
(1-p)\mathcal{J}_{\nu+1}(x)+p\frac{\mathcal{J}_{\nu+1}(x)}{\mathcal{J}_{\nu}(x)}>1>(1-q)\mathcal{J}_{\nu+1}(x)+q\frac{\mathcal{J}_{\nu+1}(x)}{\mathcal{J}_{\nu}(x)}
\end{equation}
holds for all $(x\in(0, j_{\nu,1}),$ if and only if, $p\geq\frac{\nu+1}{\nu+2}$ and $q\leq0.$\\
\noindent\textbf{Theorem D}
Let $\nu>-1$, the following inequalities 
\begin{equation}\label{mm0}
\left(1-p\right)\mathcal{I}_{\nu+1}(x)+p\frac{\mathcal{I}_{\nu+1}(x)}{\mathcal{I}_{\nu}(x)}>1>\left(1-q\right)\mathcal{I}_{\nu+1}(x)+q\frac{\mathcal{I}_{\nu+1}(x)}{\mathcal{I}_{\nu}(x)},
\end{equation}
holds for all $x\in(0,\infty)$ if and only if $p\leq\frac{\nu+1}{\nu+2}$ and $q\geq1$.\\

For $\nu>-1$ and consider the function $\mathcal{J}_{\nu}:\mathbb{R}\longrightarrow(-\infty, 1],$ defined by
\begin{equation*}
\mathcal{J}_{\nu}(x)=2^\nu\Gamma(\nu+1)x^{-\nu}J_{\nu}(x)=\sum_{n\geq}\frac{\left(\frac{-1}{4}\right)^n}{(\nu+1)_n n!}x^{2n},
\end{equation*}
where $\Gamma$ is the gamma function, $(\nu+1)_n=\Gamma(\nu+n+1)/\Gamma(\nu+1)$ for each $n\geq0$, is the well-known Pochhammer (or Appell) symbol, and $J_\nu$ defined by
$$J_\nu(x)=\sum_{n\geq0}\frac{(-1)^n(x/2)^{\nu+2n}}{n!\Gamma(\nu+n+1)},$$
stands for the Bessel function of the first kind of order $\nu$. It is worth mentioning that in
particular the function $J_\nu$  reduces to some elementary functions, like sine and cosine.
More precisely, in particular we have:
\begin{equation}\label{ee}
\mathcal{J}_{-1/2}(x)=\sqrt{\pi/2}.x^{1/2}J_{-1/2}(x)=\cos x,
\end{equation}
\begin{equation}\label{ee1}
\mathcal{J}_{1/2}(x)=\sqrt{\pi/2}.x^{-1/2}J_{1/2}(x)=\frac{\sin x}{x},
\end{equation}

For $\nu>-1,$ let us consider the function $\mathcal{I}_{\nu}:\mathbb{R}\longrightarrow[1,\infty),$ defined by
\begin{equation*}
\mathcal{I}_{\nu}(x)=2^\nu\Gamma(\nu+1)x^{-\nu}I_{\nu}(x)=\sum_{n\geq}\frac{\left(\frac{1}{4}\right)^n}{(\nu+1)_n n!}x^{2n},
\end{equation*}
where $I_\nu$ is the  modified Bessel function of the first kind defined by
$$I_\nu(x)=\sum_{n\geq0}\frac{(x/2)^{\nu+2n}}{n!\Gamma(\nu+n+1)},\; \textrm{for all}\; x\in\mathbb{R}.$$
It is worth mentioning that in particular we have
\begin{equation}\label{eee1}
\mathcal{I}_{-1/2}(x)=\sqrt{\pi/2}.x^{1/2}I_{-1/2}(x)=\cosh x,
\end{equation}
\begin{equation}\label{eee2}
\mathcal{I}_{1/2}(x)=\sqrt{\pi/2}.x^{1/2}I_{-1/2}(x)=\frac{\sinh x}{x}.
\end{equation}
\begin{equation}\label{eee3}
\mathcal{I}_{3/2}(x)=3\sqrt{\pi/2}.x^{-3/2}I_{3/2}(x)=-3\left(\frac{\sinh x}{x^3}-\frac{\cosh x}{x^2}\right).
\end{equation}

In this note, we present a generalization of the Huygens type inequalities (\ref{0001}) and (\ref{0002}) for Bessel and modified Bessel functions.

\section{\textbf{Lemmas}}

In order to establish our main results, we need several lemmas, which we present in this
section.

\begin{lemme}\label{l1}\cite{ponn} Let $a_n$ and $b_n\;(n=0,1,2,...)$ be real numbers,  and let the power series $A(x)=\sum_{n=0}^{\infty}a_{n}x^{n}$ and $B(x)=\sum_{n=0}^{\infty}b_{n}x^{n}$ be convergent for $|x|<R.$ If $b_n>0$ for $n=0,1,..,$ and if $\frac{a_n}{b_n}$ is strictly increasing (or decreasing) for $n=0,1,2...,$ then the function $\frac{A(x)}{B(x)}$ is strictly increasing (or decreasing) on $(0,R).$
\end{lemme}
\begin{lemme}\label{l2}\cite{kh, and, pin} Let $f,g:[a,b]\longrightarrow\mathbb{R}$ be two continuous functions which are differentiable on $(a,b)$. Further, let $g^{'}\neq0$ on $(a,b).$ If $\frac{f^{\prime}}{g^{\prime}}$ is increasing (or decreasing) on $(a,b)$, then the functions $\frac{f(x)-f(a)}{g(x)-g(a)}$ and $\frac{f(x)-f(b)}{g(x)-g(b)}$ are also increasing (or decreasing) on $(a,b).$
\end{lemme}
\begin{lemme}(Tur\'an type inequality for modified Bessel function)
The following Tur\'an type inequality
\begin{equation}\label{TK} 
\mathcal{I}_{\nu}(x)\mathcal{I}_{\nu+2}(x)<\frac{\nu+2}{\nu+1}\;\mathcal{I}_{\nu+1}^2(x),
\end{equation}
holds for all $\nu>-1$ and $x\in\mathbb{R}.$ In particular, the following Tur\'an type inequality
\begin{equation}\label{mama}
\cosh(x)\left(\cosh(x)-x\sinh(x)\right)<\sinh^{2}(x)
\end{equation}
is valid for all $x\in\mathbb{R}.$
\end{lemme}
\begin{proof} By using the Cauchy product
\begin{equation}\label{cauchy}
\mathcal{I}_{\mu}(x)\mathcal{I}_{\nu}(x)=\sum_{n\geq0}\frac{\Gamma(\nu+1)\Gamma(\mu+1)\Gamma(\nu+\mu+2n+1)x^{2n}}{2^{2n}\Gamma(n+1)\Gamma(\nu+\mu+n+1)\Gamma(\mu+n+1)\Gamma(\nu+n+1)}
\end{equation}
we have 
$$\frac{\nu+2}{\nu+1}\mathcal{I}^{2}_{\nu+1}(x)-\mathcal{I}_{\nu+2}(x)\mathcal{I}_{\nu}(x)=\sum_{n\geq0}\frac{\Gamma(\nu+1)\Gamma(\nu+3)\Gamma(2\nu+2n+3)}{2^{2n}\Gamma(n+1)\Gamma(2\nu+n+3)\Gamma(\nu+n+2)\Gamma(\nu+n+3)}x^{2n}\geq0$$
for all $x\in\mathbb{R}$ and $\nu>-1.$ On the other hand, observe that using (\ref{ee}), (\ref{ee1}), and (\ref{TK}) in particular $\nu=-1/2,$ the Tur\'an type inequality (\ref{TK}) becomes (\ref{mama}).
\end{proof}
\section{\textbf{Main results}}

We first obtain the further result concerning the generalized Huygens inequality to the Bessel functions described as Theorem \ref{t1}.

\begin{theorem}\label{t1} Let $-1<\nu\leq0$ and  let $j_{\nu,1}$ the first positive zero of the Bessel function $J_\nu$ of the first kind. Then the Huygens types inequalities 
\begin{equation}\label{1}
1> \left(1-\frac{\nu+2}{\nu+1}\right)\mathcal{J}_\nu (x)+\frac{\nu+2}{\nu+1}\frac{\mathcal{J}_{\nu}(x)}{\mathcal{J}_{\nu+1}(x)}
\end{equation}
holds for all $x\in(0,j_{\nu,1})$
\end{theorem}
\begin{proof} Let $\nu>-1$, we consider the function 
$$F_\nu(x)=\frac{1-\mathcal{J}_{\nu}(x)}{\frac{\mathcal{J}_{\nu}(x)}{\mathcal{J}_{\nu+1}(x)}-\mathcal{J}_{\nu}(x)},\;\;0<x<j_{\nu,1}.$$
For $0<x<j_{\nu,1}$, let 
$$f_{\nu,1}(x)=1-\mathcal{J}_{\nu}(x)\;\textrm{and}\;f_{\nu,2}(x)=\frac{\mathcal{J}_{\nu}(x)}{\mathcal{J}_{\nu+1}(x)}-\mathcal{J}_{\nu}(x).$$
Now, by again using the differentiation formula
\begin{equation}\label{555}
\mathcal{J}_{\nu}^\prime(x)=-\frac{x}{2(\nu+1)}\mathcal{J}_{\nu+1}(x) 
\end{equation}
we get 
$$f_{\nu,1}^{\prime}(x)=\frac{x\mathcal{J}_{\nu+1}(x)}{2(\nu+1)},$$
and
$$f_{\nu,2}^{\prime}(x)=\frac{x\mathcal{J}_{\nu+1}(x)}{2(\nu+1)}+\frac{\frac{x}{2(\nu+2)}\mathcal{J}_{\nu}(x)\mathcal{J}_{\nu+2}(x)-\frac{x}{2(\nu+1)}\mathcal{J}_{\nu+1}^2(x)}{\mathcal{J}_{\nu+1}^2(x)}.$$
Thus
\begin{equation}\label{888}\frac{f_{\nu,1}^{\prime}(x)}{f_{\nu,2}^{\prime}(x)}=\frac{1}{1+\frac{1}{\mathcal{J}_{\nu+1}(x)}\left(\frac{(\nu+1)\mathcal{J}_{\nu}(x)\mathcal{J}_{\nu+2}(x)}{(\nu+2)\mathcal{J}_{\nu+1}^2(x)}-1\right)}.
\end{equation}
Let 
$$h_\nu(x)=\frac{(\nu+1)\mathcal{J}_{\nu}(x)\mathcal{J}_{\nu+2}(x)}{(\nu+2)\mathcal{J}_{\nu+1}^2(x)}-1.$$
From the Tur\'an type inequality \cite{arb1}
\begin{equation}\label{5555}
\mathcal{J}_{\nu+1}^2(x)-\mathcal{J}_{\nu}(x)\mathcal{J}_{\nu+2}(x)>0.
\end{equation}
where $\nu>-1$ and $x\in(-j_{\nu,1},j_{\nu,1}),$ we conclude that $h_\nu(x)\leq0$ for all $x\in(0,j_{\nu,1}).$ On the other hand, differentiation again and simplifying give 
\begin{equation}\label{999}
h_\nu^{\prime}(x)=\frac{(\nu+1)x}{(\nu+2)\mathcal{J}_{\nu+1}^4(x)}\left[\mathcal{J}_{\nu+1}(x)\mathcal{J}_{\nu+2}(x)\left(\frac{\mathcal{J}_{\nu}(x)\mathcal{J}_{\nu+2}(x)}{\nu+2}-\frac{\mathcal{J}_{\nu+1}^2(x)}{2(\nu+1)}\right)-\frac{\mathcal{J}_{\nu}(x)\mathcal{J}_{\nu+1}^3(x)\mathcal{J}_{\nu+3}(x)}{2(\nu+3)}\right].
\end{equation}
By (\ref{5555}) and (\ref{999}) we easily get
$$h_\nu^{\prime}(x)\leq\frac{x\nu\mathcal{J}_{\nu+1}(x)\mathcal{J}_{\nu+2}(x)}{2(\nu+2)^2}-\frac{x\mathcal{J}_{\nu}(x)\mathcal{J}_{\nu+2}(x)\mathcal{J}_{\nu+3}(x)}{2(\nu+2)(\nu+3)\mathcal{J}_{\nu+1}(x)}.$$
In fact, since the function $\nu\mapsto\mathcal{J}_{\nu}(x)$ is increasing (\cite{arb1}, Theorem 3) on $(-1,\infty)$ for all fixed $x\in(-j_{\nu,1},j_{\nu,1}),$ and $\mathcal{J}_{\nu}(x)\in(0,1],$ we conclude that, by (\ref{999}), the function $h_\nu^{\prime}(x)$ is decreasing on $(0,j_{\nu,1}),$ for all $\nu\in(-1,0].$ Indeed, the function $x\mapsto\mathcal{J}_{\nu}(x)$ is decreasing on $[0,j_{\nu,1})$ (\cite{arb1}, Theorem 3) and nonnegative, which implies that the function $x\mapsto\frac{1}{\mathcal{J}_{\nu+1}(x)}\left(\frac{(\nu+1)\mathcal{J}_{\nu}(x)\mathcal{J}_{\nu+2}(x)}{(\nu+2)\mathcal{J}_{\nu+1}^2(x)}-1\right)$ is decreasing on $(0,j_{\nu,1}),$ for all $\nu\in(-1,0],$ as a product of  two functions one  is increasing and nonnegative and other is decreasing and negative. So, the function $x\mapsto\frac{f_{\nu,1}^{\prime}(x)}{f_{\nu,2}^{\prime}(x)}$ is increasing $(0,j_{\nu,1}),$ for all $\nu\in(-1,0],$ and consequently the function $x\mapsto F_\nu(x)$ is increasing $(0,j_{\nu,1}),$ for all $\nu\in(-1,0],$ by Lemma \ref{l2}. Using L'Hospital rule and (\ref{888}) yields
$$\lim_{x\rightarrow0}F_\nu(x)=\frac{\nu+2}{\nu+1}.$$
Finally, for each $\nu>-1$ and $x\in(0,j_{\nu,1})$ one has $0\leq\mathcal{J}_{\nu}(x)\leq1$, and with this the proof of inequality (\ref{1}) is complete.
\end{proof}
\begin{remark}
Using the relation (\ref{ee}), (\ref{ee1}) and $j_{-1/2,1}=\pi/2$ and from the extended Huygens type inequality (\ref{1}) for $\nu=-\frac{1}{2}$ we obtain the inequality (\ref{0001}).
\end{remark}

In the next Theorem, we establish the analogue of inequality (\ref{1}) involving the modified Bessel functions.

\begin{theorem}\label{t2}
Let $\nu>-1$. Then the Huygens types inequality 
\begin{equation}\label{111}
1> \left(1-\frac{\nu+2}{\nu+1}\right)\mathcal{I}_\nu (x)+\frac{\nu+2}{\nu+1}\frac{\mathcal{I}_{\nu}(x)}{\mathcal{I}_{\nu+1}(x)}
\end{equation}
holds for all $x\in(0,\infty).$
\end{theorem}
\begin{proof}
Let $\nu>-1$, we define the function $G_\nu$ on $(0,\infty)$ by 
\begin{equation*}
G_\nu(x)=\frac{1-\mathcal{I}_{\nu}(x)}{\frac{\mathcal{I}_{\nu}(x)}{\mathcal{I}_{\nu+1}(x)}-\mathcal{I}_{\nu}(x)}=\frac{g_{\nu,1}(x)}{g_{\nu,2}(x)},
\end{equation*}
where $g_{\nu,1}(x)=1-\mathcal{I}_{\nu}(x)$ and $g_{\nu,2}(x)=\frac{\mathcal{I}_{\nu}(x)}{\mathcal{I}_{\nu+1}(x)}-\mathcal{I}_{\nu}(x).$
By using the differentiation formula [\cite{wat}, p. 79]
\begin{equation}\label{mm}
\mathcal{I}_{\nu}^{\prime}(x)=\frac{x}{2(\nu+1)}\mathcal{I}_{\nu+1}(x)
\end{equation}
can easily show that
\begin{equation}\label{666}
\frac{f_{\nu,1}^{\prime}(x)}{f_{\nu,2}^{\prime}(x)}=\frac{1}{1+\frac{1}{\mathcal{I}_{\nu+1}(x)}\left(\frac{(\nu+1)\mathcal{I}_{\nu}(x)\mathcal{I}_{\nu+2}(x)}{(\nu+2)\mathcal{I}_{\nu+1}^2(x)}-1\right)}.
\end{equation}
Now, for $\nu>-1,$ we define the function $k_\nu$ by:
$$k_\nu(x)=\frac{(\nu+1)\mathcal{I}_{\nu}(x)\mathcal{I}_{\nu+2}(x)}{(\nu+2)\mathcal{I}_{\nu+1}^2(x)}-1.$$
From the Tur\'an type inequality (\ref{TK}) (see Lemma 3), we conclude that $k_\nu(x)\leq0$ for all $x\in\mathbb{R}.$ On the other  hand, using the Cauchy product \ref{cauchy},
%\begin{equation}
%\mathcal{I}_{\mu}(x)\mathcal{I}_{\nu}(x)=\sum_{n\geq0}\frac{\Gamma(\nu+1)\Gamma(\mu+1)\Gamma(\nu+\mu+2n+1)x^{2n}}{2^{2n}\Gamma(n+1)\Gamma(\nu+\mu+n+1)%\Gamma(\mu+n+1)\Gamma(\nu+n+1)}
%\end{equation}
we get
$$\frac{(\nu+1)\mathcal{I}_{\nu}(x)\mathcal{I}_{\nu+2}(x)}{(\nu+2)\mathcal{I}_{\nu+1}^2(x)}=\frac{\sum_{n=0}^{\infty}a_n x^{2n}}{\sum_{n=0}^{\infty}b_n x^{2n}},$$
where $a_n(\nu)=\frac{\Gamma^2(\nu+2)\Gamma(2\nu+2n+3)}{2^{2n}\Gamma(n+1)\Gamma(\nu+n+1)\Gamma(\nu+n+3)}$ and $b_n(\nu)=\frac{\Gamma^2(\nu+2)\Gamma(2\nu+2n+3)}{2^{2n}\Gamma(n+1)\Gamma^2(\nu+n+2)}$ for all $n=0,1,...$
So, for all $n=0,1,...,$ we have 
$$c_n(\nu)=\frac{a_n(\nu)}{b_n(\nu)}=\frac{\Gamma^2(\nu+n+2)}{\Gamma(\nu+n+1)\Gamma(\nu+n+3)}=\frac{\nu+n+1}{\nu+n+2},$$
we conclude that $c_n(\nu)$ is increasing for $n=0,1,...,$ and the function $x\mapsto k_\nu(x)$ is increasing on $(0,\infty),$ by Lemma \ref{l1}. Since the function $x\mapsto\frac{1}{\mathcal{I}_{\nu+1}(x)}$ is decreasing and nonnegative on $(0,\infty)$ and the function $x\mapsto k_\nu(x)$ is increasing and negative on $(0,\infty),$ we conclude that $x\mapsto\frac{g_{\nu,1}^{\prime}(x)}{g_{\nu,2}^{\prime}(x)}$ is decreasing on $(0,\infty),$ and consequently the function $x\mapsto G_\nu(x)$ is decreasing on $(0,\infty)$, by Lemma \ref{l1}. 
Therefore, from the L'Hospital rule and (\ref{666}) yields
$$\lim_{x\rightarrow0}G_\nu(x)=\frac{\nu+2}{\nu+1}.$$
Moreover, using the fact $\mathcal{I}(x)\geq 1$, we get the Huygens type inequality (\ref{111}). So, the proof of Theorem \ref{t2} is complete.
\end{proof}
\begin{remark}
1. From the relations (\ref{eee1}) and (\ref{eee2}) we find that the inequality (\ref{111}) is the generalization of inequality (\ref{0002}).\\
2. Since the function $x\mapsto G_\nu(x)$ is decreasing on $(0,\infty)$, and using the asymptotic formula [\cite{1}, p. 377]
\[
I_{\nu}(x)=\frac{e^{x}}{\sqrt{2\pi x}}\left[1-\frac{4\nu^{2}-1}{1!(8x)}+\frac{(4\nu^{2}-1)(4\nu^{2}-9)}{2!(8x)^{2}}-...\right]\]
which holds for large values of $x$ and for fixed $\nu > - 1$, we obtain
$$\lim_{x\longrightarrow\infty}G_{\nu}(x)=1.$$ Then, the following inequality \cite{arb1} 
$$\mathcal{I}_{\nu+1}(x)\leq\mathcal{I}_{\nu}(x),$$
holds for all $x\in\mathbb{R}$ and $\nu>-1.$\\
3. Using the relation (\ref{eee2}) and (\ref{eee3}) from the extended Huygens type inequality (\ref{111}) for $\nu=1/2,$ we obtain the following inequality
$$9>\frac{\sinh x}{x}\left(-6+\frac{5x^3}{x\cosh x-\sinh x}\right),$$
which holds for all $x\in\mathbb{R}.$
\end{remark}


\begin{thebibliography}{99}
\bibitem{1}M. Abramowitz and I. A. Stegun (eds), Handbook of Mathematical Functions with Formulas,
Graphs and Mathematical Tables (Dover Publications, New York, 1965).
\bibitem{and} G.D. Anderson, S.-L. Qiu, M.K. Vamanamurthy, M. Vuorinen, Generalized elliptic integral and modular equations, Pacific J. Math. 192 (2000) 1–37.
\bibitem{arb1} A. Baricz, Functional inequalities involving Bessel and modified Bessel functions of the first kind,
Expo. Math., 26 (2008),
%\bibitem{erd}  A. Erd\'elyi, W. Magnus, F. Oberhettinger, F. Tricomi, Higher transcendental Functions, vol. 2, McGraw-Hill,
%New York, 1954.
\bibitem{14} F.  Qi,  D.-W.  Niu,  and  B.-N.  Guo,  Refinements,  generalizations, 
and  applications  of  Jordan's  inequality  and  related  problems,  J. 
Inequal. Appl. 2009 (2009), Article ID 271923, 52 pages. 
\bibitem{H} C. Huygens, Oeuvres completes, publiees par la Societe hollandaise des science, Haga, 1888–1940 (20 volumes).
\bibitem{kh} K. Mehrez, Redheffer type Inequalities for modified Bessel functions, Arab. Jou. of Math. Sci. 2015.
\bibitem{KH} K. Mehrez, Extension of Huygens type inequalities for Bessel and modified Bessel Functions, arXiv:1512.05798v1.
\bibitem{N} E.  Neuman  and  J.  Sandor,  On  some  inequalities  involving trigonometric  and  hyperbolic  functions  with  emphasis  on  the  Cusa-Huygens,  Wilker,  and  Huygens  inequalities,  Math.  Inequal.  Appl. 13 (2010), no. 4, 715-723. 
\bibitem{pin}I. Pinelis, ''Non-strict'' l'Hospital-type rules for monotonicity: intervals of constancy, J. Inequal. Pure Appl. Math. 8 (1) (2007). article 14, 8 pp.,(electronic). 
\bibitem{ponn} S. Ponnusamy, M. Vuorinen, Asymptotic expansions and inequalities for hypergeometric functions, Mathematika 44 (1997) 278–301.
\bibitem{wat} G.N. Watson, A Treatise on the Theory of Bessel Functions, Cambridge University Press, 1922.
\bibitem{zhu7} L. Zhu, Some new inequalities of the Huygens type, Comp. and Math. with App. 58 (2009) 1180–1182
\end{thebibliography}
\end{document}